\documentclass[a4paper,11pt]{article}

\usepackage{amsmath}
  \usepackage{paralist}
  \usepackage{graphics} 
  \usepackage{epsfig} 
\usepackage{graphicx}  \usepackage{epstopdf}
 \usepackage[colorlinks=true]{hyperref}
 \usepackage{amsfonts,latexsym,amssymb} 
\usepackage{mathrsfs,amsmath,amsthm}
\usepackage[all,cmtip]{xy}
\usepackage{graphicx,float,color}
\usepackage[squaren, Gray, cdot]{SIunits}
\usepackage{mathrsfs}
\usepackage[utf8]{inputenc}

\newcommand{\vecC}{\mathbf{C}}
\newcommand{\vecD}{\mathbf{D}}

\newcommand{\vecU}{\mathbf{U}}

\newcommand{\vech}{\mathbf{h}}

\newcommand{\vecu}{\mathbf{u}}

\newcommand {\Div}  {\mbox{\rm div\,}}

\newtheorem{theorem}{Theorem}[section]
\newtheorem{corollary}{Corollary}
\newtheorem{lemma}[theorem]{Lemma}
\newtheorem{proposition}{Proposition}

\newtheorem{definition}[theorem]{Definition}
\newtheorem{remark}{Remark}

\author{Marta Zoppello}

\begin{document}
\title
      {Mean-field and kinetic limit of affine control systems: optimal control through leaders}\date{ }
%
\maketitle

\begin{center}
Marta Zoppello\footnote{Politecnico di Torino, Corso Duca degli Abruzzi, 24
10129 Torino, Italy, marta.zoppello@polito.it}\end{center}
%
%
\begin{abstract}
This paper studies a multi-agent system starting from a single agent dynamics which is a nonlinear affine control system. It analyze what happens when the number of agents goes to infinity using two different approaches,  a granular mean-field one, trying to control only a finite number of agents while all the others goes to infinity, and a kinetic one, when in the limit the percentage of controlled agents is preserved. In both cases the optimal control problem starting from the solution for the finite dimensional one is stated.
\end{abstract}

\section{Introduction}

In recent years there has been a growing interest in literature in multi-agents systems \cite{CuckerDong11,CuckerSmale07,Vicsek} for example describing active cells \cite{Koch,Pertame}, cooperative robots \cite{ChuangDOrsogna07,PereaGomez} or animal motion \cite{CarrilloDOrsogna09,CristianiPiccoliTosin10, CristianiPiccoliTosin11}. Most of these models develop from a single agent dynamics adding an interaction term. The main goal of these studies is to clarify the relationship between the interplay of simple binary interaction forces, and the potential emergence of a global behavior in the form of specific patterns, such as the formation of flocks or swarms in animal motion.\\ 
However for most of the models, the description of the asymptotic behavior of a very large system of agents can become an hard task. A classical way to approach the global description of the system is then to focus on its mean behavior, as in the classical mean-field theory \cite{CarrilloChoi}. In certain circumstances, the formation of a specific pattern is conditional to the position of the initial datum, i.e if it belongs to a certain basin of attraction we can characterize certain pattern formation, otherwise we are not able to guarantee it. Thus it is interesting to wonder whether an external agent  can intervene on the system towards pattern formation, also in those situations where this phenomenon is not the result of autonomous self-organization. This intervention may be modeled as an additional control which can be either an external forcing term or the direct control of some degrees of freedom of some agents. The standard way, consists in regarding ``external forces'' as controls, modeled usually as an additional vector field in the Newton equations, this problem has been studied for example in \cite{AlbiPareschi15,AlbiFornasier16,FornasierPiccoli14}.  A less explored strategy is instead to think that the controller acts on the system by directly assigning the values of the coordinates of some of the agents, regarded as control parameters. This can happen for example in the collective motion of the so called \textit{robotic locomotion systems}, introduced in \cite{FassoZoppello}, where even the single agent dynamics is controlled assigning the evolution of some coordinates.\\
In this paper we follow this latter approach focusing our attention on systems characterized by a binary interaction which depends on the velocity of the agents. This approach allows us to directly control the velocity of the coordinates describing some of the agents, ending up with a nonlinear finite dimensional system affine in these controls. This is the case for example of spherical particles immersed in a Stokes fluid \cite{Hocking,Moreau21}, which can be considered a first step for  modeling a large number of micro-robots swimming in a highly viscous fluid. \\
The aim of this paper is to study what happens when the number of agents goes to infinity.
Although we still use a binary interaction approach, we extend the techniques developed in \cite{AlbiPareschi15,FornasierPiccoli14} to systems with a different type of dynamics, namely nonlinear affine control system. This is a less investigated type of dynamics, but very interesting for certain applications, for example micro-robots swimming, swarms, crowd dynamics, vehicular traffic, opinion formation, wealth distribution and many others, when the interaction between agents is generated by their velocities.\\
In this work we analyze two different situations, the fist one in which the system is composed by a small number of active controlled agents, usually called leaders \cite{FornasierPiccoli14}, that try to influence the dynamics of a huge number of passive agents, known as followers. This happens for example in opinion dynamics when a small group of people try to pilot the opinion of a large crowd, or in biology when only few active cells are able to influence a huge population. Opinion dynamics in the presence of different
populations has been previously introduced in \cite{BordognaAlbano,Totzek,WeisbuchDeffuant}. We mention here that control
through leaders in self-organized flocking systems has been studied in \cite{AlbiPareschi13, WongkaewBorzi} as well as in crowd dynamics in \cite{AlbiBongini16}. The technical derivation of the mean-field optimal control is realized by the simultaneous development of
the mean-field limit of the equations governing the followers dynamics, while the leaders dynamics remains finite dimensional, together with the $\Gamma$-limit of the finite dimensional optimal control problems. Moreover we also derive the
corresponding first order optimality conditions, resulting in a coupled system of forward/backward time-dependent  ODE's for the finite dimensional optimal control problem or PDEs for the infinite dimensional one.\\
The second situation instead is when the overall number of agents goes to infinity, preserving the percentage of controlled ones. This is the case when the number of controls should be a fraction of the total number of agents in order to be effective. We develop a Boltzmann-type control approach following the ideas recently presented in \cite{Albi16}. It yields an approximation of the mean-field dynamics by means of an iterative sampling of 2-agent microscopic dynamics (binary dynamics). This same principle allows us to generate control signals for the mean-field model by means of solving optimal control problems associated to the binary dynamics.

\section{The finite dimensional system of equations and mathematical preliminaries}
In several mechanical control systems, called ``hard-device'' systems, the controls are precisely given by further degrees of freedom and the evolution of the remaining coordinates can be determined by solving the control equations.
More precisely especially in systems driven by viscous forces, like for example spherical particles immersed in Stokes fluid, we can suppose to control directly the velocity of some of them, denoted by ``leaders'' or ``active agents'' , and determine the evolution of the ``followers''  or``passive agents''  through their interaction with the active ones. This interaction, due to the viscosity of the system, is linear in the velocity of the active agents, therefore it is natural that in the equations the control appears both in the leaders equations and also in the followers one, and  turn out to be a finite dimensional dynamics affine in the $m$ controls \cite{Moreau21} .
In this paper we focus on this kind of multi-agents systems in which the individual dynamics is an ``hard-device'' control system, the overall finite-dimensional  dynamics  is in $d\times (N+m)$ variables,  where $N$ is  the  number  of passive agents, $m$ are the controlled ones, and $d$ is the dimension of the space in which the motion of such individuals evolves. More precisely the equation of motion of the $N+m$ agents are
\begin{equation}
\label{ODE_N}
\small
\begin{cases}
\dot Y_i=\vecu_i\qquad i=1,\cdots m\\
\dot X_i=\frac{1}{N}\sum_{j=m+1}^NH(X_i-X_j)+\sum_{\ell=1}^d\frac{1}{m}\sum_{k=1}^mG^\ell(X_i-Y_k)u_{k\,\ell}\quad i=1\cdots N
\end{cases}
\end{equation}
where $Y_i\in\mathbb{R}^d$ is the state of the $i$-th controlled agent, $X_i\in\mathbb{R}^d$ represents the state of the $i-$th passive agent and the control $\vecu_i=(u_{i\,1},\,\cdots,\,u_{i\,d})\in\mathbb{R}^d$ is the velocity of the $i$-th active agent for $i=1,\cdots m$ and are measurable functions of time. Moreover $H$ and $G^\ell$ are vectored valued functions: $H$ represents the interaction between non-controlled agents and $G^\ell$ $\ell=1\cdots d$ account for the interaction that the non-controlled agent $i$ feels when the $k$-th controlled agent moves with a velocity $u_{k\ell}$ along the $\ell$-th direction of $\mathbb{R}^d$.
Moreover we make the following hypothesis on $H$ and $G^\ell$
\begin{itemize}
\item[(H)] $H:\mathbb{R}^d\to\mathbb{R}^d$ and $G^\ell: \mathbb{R}^d\to\mathbb{R}^d$ for $\ell=1,\cdots, d$ are, bounded, locally Lipschitz and with sub-linear growth, i.e.
 $$
|H(\xi)|\leq C_H(1+|\xi|)\qquad|G^\ell(\xi)|\leq C^\ell_G(1+|\xi|)\,,
 $$
 where $|\cdot|$ is the euclidean norm in $\mathbb{R}^d$ and $\xi=(\xi_1,\cdots,\xi_d)$.
\end{itemize}
Furthermore we are interested in solving the following optimal control problem
$$
\min_{(\vecu_1,\cdots, \vecu_m)\in L^{\infty}([0,T],U)}\int_0^T \Bigg(L(Y(t),X(t))+\sum_{k=1}^m|\vecu_k(t)|^2\Bigg)\,dt\,,
$$
with $\mathcal{U}\subset B(0,U)\,U>0$ a  a fixed nonempty compact subset of $\mathbb{R}^{md}$ and $(Y,X)$ subjected to the evolution equation \eqref{ODE_N}.
For example 
$$
L(Y(t),X(t)):=\frac{1}{2} \sum_{i=1}^N|X_i(t)-X^*|^2,
$$
where $X^*$ is a desired configuration of agents, can be used if we want to minimize the distance from a fixed state configuration.\\
Our aim is to see what happens to these equations when the number of agents tents to infinity. We will treat different cases, more precisely
\begin{itemize}
\item[a)] Only the  number of non-controlled agents tends to infinity and $m$ remains finite. In this case we use a mixed granular mean-field approach. 
\item[b)] $N\to\infty$ preserving the percentage of controlled agents. In this case we will use the kinetic limit which tries to preserve the percentage of controlled agents
\end{itemize}
\subsection{Mathematical preliminaries}
In this section we present some basic definitions and some mathematical results that we will need to prove the main theorems. \\
In the following we consider the space $\mathcal{P}_1(\mathbb{R}^n)$, consisting of all probability measures on $\mathbb{R}^n$ of finite first  moment.  If  we  denote  $\Pi(\mu,\nu)$ the set of probability measures on $\mathbb{R}^n\times\mathbb{R}^n$ with first and second marginals equal to $\mu$ and $\nu$ respectively, then we have that the \textit{Wasserstein distance} between the measures $\mu$ and $\nu$ is given by
\begin{equation}
\label{Waaserstein}
\mathcal{W}_1(\mu,\nu):=\inf_{\pi\in\Pi(\mu,\nu)}\Bigl\{\int_{\mathbb{R}^n\times\mathbb{R}^n}|x-y|d\pi(x,y)\Bigr\}.
\end{equation}
Let us define the following 
\begin{equation}
\label{empirical_mes}
\mu_{N}=\frac{1}{N}\sum_{j=1}^N \delta_{X_j}\qquad \mu_{m\,\ell}=\frac{1}{m}\sum_{k=1}^mu_{k\,\ell}\delta_{Y_k}, \quad\ell=1,\cdots \ell\,.
\end{equation}
with $\vecu\in\mathcal{U}$,  a nonempty compact subset of $\mathbb{R}^{md}$.\\
Our system of equations \eqref{ODE_N} can be rewritten as
\begin{equation}
\label{eq_convolution1}
\begin{cases}
\dot Y_k=\vecu_k\quad k=1\ldots m\\
\dot X_i=H*\mu_{N}(X_i)+\sum_{\ell=1}^dG^\ell*\mu_{m\,\ell}(X_i)\quad i=1\ldots N\,.
\end{cases}
\end{equation}
Therefore we have the following proposition
\begin{proposition}
\label{sub_growth}
Given $H$ and $G^\ell$ $\ell=1,\cdots, d$ satisfying hypothesis (H), and $\mu_{N}$, $\mu_m$ defined in \eqref{empirical_mes}, we have
$$
\begin{aligned}
&|H*\mu_{N}(X_i)|\leq C_H\bigl(1+|X_i|+\frac{1}{N}\sum_{j=1}^N|X_j|\bigr)\,,\\
&\quad |G^\ell*\mu_{m}(X_i)|\leq C^\ell_G\bigl(1+|X_i|+\frac{1}{m}\sum_{k=1}^m|u_{k\,\ell}||Y_k|\bigr)\leq C^\ell_G \bigl(1+|X_i|+\frac{1}{m}\sum_{k=1}^m|Y_k|\bigr)\,.
\end{aligned}
$$
\end{proposition}
\begin{proof}
It follows immediately from the sub-linear growth (H) of $H$ and $G^\ell$ for $\ell=1,\cdots d$.
\end{proof}
\noindent Moreover if $\zeta=(X,Y)\in\mathbb{R}^{dN}\times\mathbb{R}^{d\,m}$, we denote by
\begin{equation}
\label{finite_norm}
||\zeta||:=||Y||_m+||X||_{N}:=\frac{1}{m}\sum_{i=k}^m |Y_k|+\frac{1}{N}\sum_{i=1}^N|X_i|\,.
\end{equation}
We can therefore state the following existence and uniqueness of the solution
\begin{proposition}
Given a control function $\vecu=(\vecu_1,\ldots \vecu_m)\in L^\infty([0,T],\mathbb{R}^{md})$ and an initial datum $\bar{\zeta}=(\, \bar{X}_{1},\ldots,\bar{X}_N,\bar{Y}_1,\ldots, \bar{Y}_m)$ there exists a unique Caratheodory solution $\zeta(t)=(X_{1}(t),\ldots,X_N(t),\,Y_1(t),\ldots Y_m(t))$ of \eqref{eq_convolution1} such that
\begin{equation}
\label{2_6}
\int_0^T||\zeta(t)||\leq (||\bar{\zeta}||+\tilde{C}T)e^{\tilde{C}T}\,.
\end{equation}
Moreover, the trajectory is Lipschitz continuous in time
\begin{equation}
\label{2_7}
||\zeta(t_1)-\zeta(t_2)||\leq\mathcal{L}|t_1-t_2|\,.
\end{equation}
\end{proposition}
\begin{proof}
Thanks to the explicit form \eqref{eq_convolution1}, to Proposition \ref{sub_growth} and by the boundedness of $\vecu_k$ the right-hand side of \eqref{eq_convolution1} fulfills the linear growth condition. Thus the solution is uniquely determined moreover we have
$$
||\zeta(t_1)-\zeta(t_2)||\leq  \int_{t_1}^{t_2}\tilde{C}(1+||\zeta||)\,ds\leq\tilde{C}(1+||\zeta_0||+\tilde{C}T)e^{\tilde{C}T}|t_1-t_2|\,.
$$
Thus the conclusion follows.
\end{proof}
\noindent Furthermore we have also the continuous dependence on the data holds also in the Caratheodory setting
 \begin{lemma}
\label{lemma_gronwal}
Let $g_1$ and $g_2$ be Caratheodory functions both satisfying
$$
|g_i(t,y)|\leq C(1+|y|)\qquad i=1,2.
$$
Let $r>0$ and define
$$
\rho_{r,m,T}:=\bigl(r+CT\bigr)e^{CT}\,,
$$
assume also that $g_1$ is Lipschitz function with respect to the $y$ variable, for all $|y_i|\leq\rho_{r,m,T} \,, i=1,2$\,. Set 
$$
q(t):=||g_1(t,\cdot)-g_2(t,\cdot)||_{L^{\infty}(B(0,\rho_{r,m,T}))}\,.
$$
Then if $\dot y_1=g_1(t,y_1)$ and $\dot y_2=g_2(t,y_2)$, $|y_1(0)|\leq r$ and $|y_2(0)|\leq r$, one has
$$
y_1(t)-y_2(t)\leq e^{Lt}|y_1(0)-y_2(0)|+\int_0^t e^{\int_s^t L\,d\sigma}q(s)\,ds\,.
$$
\end{lemma}
\begin{proof}
It follows from \cite{FornasierPiccoli14} and Gronwall's lemma
\end{proof}
\noindent Here we recall the statement and proof of Lemmas and Theorems concerning Lipschitz continuity estimates for transport flows induced by Caratheodory dynamics, that are an adjustment to the ones in \cite{Fornasier14,FornasierPiccoli14} to the dynamics \eqref{eq_convolution1} which is affine in the controls. They are preparatory for the main theorems of the next sections regarding the existence of a solution in the limit for the number of agents that tend to infinity.
\begin{lemma}
\label{lemma_lipschitz}
Let $H: \mathbb{R}^d\to \mathbb{R}^d$ and $G^\ell: \mathbb{R}^d\to \mathbb{R}^d$, $\ell=1,\cdots,d$ be locally Lipschitz functions such that 
$$
|H(y)|\leq C_H(1+|y|), \qquad |G^\ell(y)|\leq C_{G^\ell}(1+|y|)\quad \text{for $\ell=1,\cdots,d$ and for all $y\in\mathbb{R}^d$}\,,
$$
moreover let $\mu:[0,T]\to \mathcal{P}_1(\mathbb{R}^d)$ be a continuous map with respect to $\mathcal{W}_1$, and $ \mu_{m\,\ell}:[0,T]\to \mathcal{P}_1(\mathbb{R}^d)$, time-dependent  measure  supported  on  the  respective absolutely  continuous  trajectory $t\mapsto Y_k(t)$, i.e  $\mu_{m\,\ell}=\frac{1}{m}\sum_{k=1}^m u_{k\,\ell}\, \delta_{Y_k}$ with $u_k\in L^{\infty}[0,T]$. Then there exist  constants $C^{'}_H$ and $C^{'}_{G^\ell}$ such that 
$$
\begin{aligned}
&|H*\mu(t)(y)|\leq C^{'}_H(1+|y|)\,,\\
&|G^\ell*\mu_{m,\,\ell}(t)(y)|\leq C^{'}_{G^\ell}(1+|y|)\,\, \text{for $\ell=1,\cdots,d$ and all $t\in[0,T]$.}
\end{aligned}
$$
Furthermore, if
$$
\mathrm{supp}\,\, \mu(t)\subset B(0,R),
$$
for all $t\in[0,T]$, then for every compact subset $K\subset \mathbb{R}^d$ there exist a constants $L_{{R,K}_{H}}$ $L_{{R,K}_{G^\ell}}$ such that
$$
\begin{aligned}
&|H*\mu(t)(y_1)-H*\mu(t)(y_2)|\leq L_{{R,K}_{H}}|y_1-y_2|\,,\\
&|G^\ell*\mu_{m,\,\ell}(t)(y_1)-G^\ell*\mu_{m,\,\ell}(t)(y_2)|\leq L_{R,K_{G^\ell M}}|y_1-y_2|\,,
\end{aligned}
$$
where $M$ is the maximum of the norms of the $\vecu_k$. 
For every $t\in[0,T]$ and every $y_1,\,y_2\in K$.
\end{lemma}
\begin{proof}
$$
\begin{aligned}
|H*\mu(t)(y)|=|&\int_{\mathbb{R}^d}H(y-\xi)\,d\mu(\xi)|\leq  C_H(1+|y|)+\int_{\mathbb{R}^d}|\xi|\,d\mu(\xi)\\
&\leq C^{'}_H(1+|y|)\,,
\end{aligned}
$$
where in the last inequality we have used the uniform boundedness of the last term thanks to our continuity assumption. Analogously
$$
\begin{aligned}
|G^\ell*\mu_{m\,\ell}(t)(y)|=|&\int_{\mathbb{R}^d}G^\ell(y-\xi)\,d\mu_{m\,\ell}(\xi)|\leq  C_{G^\ell}(1+|y|+\frac{1}{m}\sum_{k=1}^m |u_{k\,\ell}||\xi_k|)\\
&\leq C^{'}_{G^\ell M}(1+|y|)\,,
\end{aligned}
$$
where for the last inequality we have used the fact that $u\in L^{\infty}$ and the uniform boundedness of the sum.\\
Finally if $\mathrm{supp}\,\, \mu(t)\subset B(0,R),$ we have
$$
\begin{aligned}
&|H*\mu(t)(y_1)-H*\mu(t)(y_2)|=|\int_{\mathbb{R}^d}\left(H(y_1-\xi)-H(y_2-\xi)\right)\,d\mu(\xi)|\leq L_{{R,K}_{H}}|y_1-y_2|\,,\\
& |G^\ell*\mu_{m\,\ell}(t)(y_1)-G^\ell*\mu_{m\,\ell}(t)(y_2)|=|\int_{\mathbb{R}^d}\left(G^\ell(y_1-\xi)-G^\ell(y_2-\xi)\right)\,d\mu_{m\,\ell}(\xi)|\\
&\phantom{ |G^\ell*\mu_m(t)}\leq \frac{1}{m}\sum_{k=1}^m|u_{k\,\ell}||G^\ell(y_1-\xi_k)-G^\ell(y_2-\xi_k)|\leq\frac{1}{m}\sum_{k=1}^m|u_{k\,\ell}|L_{R,K_{G^\ell}}|y_1-y_2|\\
&\phantom{ |G^\ell*\mu_m(t)(y_1)|}\leq L_{{R,K_{G^\ell}}}|y_1-y_2|\,.
\end{aligned}
$$
\end{proof}
\noindent The following regards the regularity with respect to the initial condition
\begin{lemma}
\label{lemma_trasporto}
Let $H: \mathbb{R}^d\to \mathbb{R}^d$ and $G^\ell: \mathbb{R}^d\to \mathbb{R}^d$  $\ell=1,\cdots, d$  be locally Lipschitz functions such that 
$$
|H(y|\leq C_H(1+|y|), \qquad |G^\ell(y)|\leq C_{G^\ell}(1+|y|)\quad \text{for all $y\in\mathbb{R}^d$}\,,
$$
and let $\mu^1:[0,T]\to\mathcal{P}_1(\mathbb{R}^d)$ and $\mu^2:[0,T]\to\mathcal{P}_1(\mathbb{R}^d)$ be continuous maps with respect to $\mathcal{W}_1$, both satisfying
$$
\mathrm{supp}\,\, \mu^1(t)\subset B(0,R),\qquad \mathrm{supp}\,\, \mu^2(t)\subset B(0,R),
$$
and $\mu^1_m\,,\, \mu^2_m$ two time-dependent measures supported on the respective absolutely continuous trajectories $t\mapsto X^i_k(t)$, i.e  $\mu_{m\,\ell}=\frac{1}{m}\sum_{k=1}^m u^i_{k\,\ell} \,\delta_{X^i_k}\,, i=1,2$ with $u^i_{k\,\ell}\in L^{\infty}[0,T]$. Consider the flow map $\mathcal{T}^{\mu^i,\mu_m^i},i=1,2$ associated to the system
\begin{equation}
\label{dyn}
\dot X=H*\mu^i(t)(X(t))+\sum_{\ell=1}^dG^\ell*\mu^i_{m\,\ell}(t)(X(t))\,.
\end{equation}
Fix $r>0$: then there exist a constant $\rho$, a constant $L>0$ both depending only on $r,\, C,\, R,\, T$ and $M=\max_{\ell,k,i}||u^{i^\ell}_k||_{L^{\infty}}$ such that
$$
\begin{aligned}
|\mathcal{T}^{\mu^1,\mu_{m\,\ell}^1}_t&(P_1)-\mathcal{T}^{\mu^2,\mu_{m\,\ell}^2}_t(P_2)|\leq e^{Lt}|P_1-P_2| \\
&+\int_0^t e^{L(t-s)}||H*\mu^1(s)-H*\mu^2(s)+\sum_{\ell=1}^d(G^\ell*\mu^1_{m\,\ell}(s)-G^\ell*\mu^2_{m\,\ell}(s))||_{L^{\infty}}\,,
\end{aligned}
$$
whenever $|P_1|\leq r$ and $|P_2|\leq r$.
\end{lemma}
\begin{proof}
Let $g_i:[0,T]\times \mathbb{R}^d\to\mathbb{R}^d$ be the right hand side of \eqref{dyn}. As in Lemma \ref{lemma_lipschitz} we can find a constant $\bar{C}$ such that 
$$
|H*\mu^i(t)(P)|+\sum_{\ell=1}^d|G^\ell*\mu^i_{m\,\ell}(t)(P)|\leq \bar{C}(1+|P|)\,.
$$
Therefore, for every $P_1,\, P_2\in\mathbb{R}^d$ such that $|P_i|\leq r$ applying Lemma \ref{lemma_gronwal} we have
$$
|\mathcal{T}^{\mu^i,\mu_{m\,\ell}^i}_t(P_i)|\leq\bigl(r+\bar{C}^{'}T\bigr)e^{\bar{C}^{'}T}\,,
$$
and setting $\rho:=\bigl(r+\bar{C}^{'}T\bigr)e^{\bar{C}^{'}T}$ we have
$$
\begin{aligned}
||g_1(t,\cdot)-g_2(t,\cdot)||_{L^{\infty}(B(0,\rho))}=&||H*\mu^1(s)-H*\mu^2(s)+\\
&\sum_{\ell=1}^d(G^\ell*\mu^1_{m\,\ell}(s)-G^\ell*\mu^2_{m\,\ell}(s))||_{L^{\infty}(B(0,\rho))}\,.
\end{aligned}
$$
\end{proof}
\noindent The following lemma gives a result on the transportation of a measure by two different functions
\begin{lemma}
\label{lemma66}
Let $E_1$ and $E_2:\mathbb{R}^d\to\mathbb{R}^d$ be two borel measurable functions. Then, for every $\mu\in\mathcal{P}_1(\mathbb{R}^d)$ one has
$$
\mathcal{W}_1((E_1)_{\sharp}\mu,(E_2)_{\sharp}\mu)\leq||E_1-E_2||_{L^\infty(\mathrm{supp}\mu)}.
$$
If in addition $E_1$ is locally Lipschitz continuous, and $\mu,\,\nu\in\mathcal{P}_1(\mathbb{R}^n)$ are both compactly supported on a ball then
$$
\mathcal{W}_1((E_1)_{\sharp}\mu,(E_1)_{\sharp}\nu)\leq L_r \mathcal{W}_1(\mu,\nu)\,.
$$
\end{lemma}
\begin{proof}
It can be found in \cite{FornasierPiccoli14}.
\end{proof}

\section{Mean-field limit}
In this section we will derive a system of equations governing the dynamics of an infinite number of agents when only a finite number $m$ of them is controlled. In particular, differently from what already done in literature, we focus of the case when the finite dimensional dynamics turns out to be a nonlinear affine control system.\\
As $N\to\infty$ we claim that the ODEs system \eqref{eq_convolution1} turns to
\begin{equation}
\label{eq_mu}
\begin{cases}
Y_k=\vecu_k\qquad \text{for}\,i=1\ldots m\\
\partial_t\mu+\nabla\cdot\Bigl[\bigl(H*\mu+\sum_{\ell=1}^dG^\ell*\mu_{m\,\ell}\bigr)\mu\Bigr]=0\,.
\end{cases}
\end{equation}
\begin{definition}
\label{def_sol}
Let $u\in L^\infty([0,T],\mathcal{U})$ be given. We say that a map $(Y,\mu):[0,T]\to\chi:=\mathbb{R}^{dm}\times\mathcal{P}_1(\mathbb{R}^d)$ is a solution of the controlled system \eqref{eq_mu} if
\begin{itemize}
\item The measure $\mu$ is equi-compactly supported in time, i.e., there exists $R>0$ such that $\mathrm{supp}(\mu(t))\subset B(0,R)$ for all $t\in[0,T]$
\item the solution is continuous in time with respect to the following metric
$$
||(Y,\mu)-(Y^{'},\mu^{'})||_{\chi}:=\frac{1}{m}\sum_{k=1}^m|Y_k-Y_k^{'}|+\mathcal{W}_1(\mu,\mu^{'})\,,
$$
where $\mathcal{W}_1$ is the $1$-Wasserstein distance
\item The $Y$ coordinates define a Caratheordory solution of the following controlled problem with control $u$
$$
\dot Y_k=\vecu_k\,.
$$
\item The $\mu$ component satisfies 
$$
\frac{d}{dt}\int_{\mathbb{R}^d}\varphi(x)\,d\mu(t)(x)=\int_{\mathbb{R}^d}\nabla\varphi\cdot \bigl(H*\mu(t)(x)+\sum_{\ell=1}^dG^\ell*\mu_{m\,\ell}(t)(x)\bigr)\,d\mu(t)(x)\,,
$$
for every $\varphi\in C^\infty_c(\mathbb{R}^d)$.
\end{itemize}
\end{definition}
\noindent Fist of all  it will be convenient to address the stability of the system \eqref{eq_mu}.\\
Let us assume the following:
\begin{itemize}
\item [(A)] Let $H:\mathbb{R}^d\to\mathbb{R}^d$ and $G^\ell:\mathbb{R}^d\to\mathbb{R}^d, \ \ell=1,\cdots d$ be locally Lipschitz functions with sub-linear growth, let $\mu:[0,T]\to\mathcal{P}_1(\mathbb{R}^d)$ and $\nu:[0,T]\to\mathcal{P}_1(\mathbb{R}^d)$ be continuous maps with respect to $\mathcal{W}_1$ both satisfying
$$
\mathrm{supp} \mu(t)\subset B(0,R)\qquad\mathrm{supp} \nu(t)\subset B(0,R),
$$
Let $\vecu\in L^{\infty}$ be fixed and $\mu_{m\ell}=\frac{1}{m}\sum_{k=1}^m u_{k\ell} \delta_{Y_k}$ and $\mu^{'}_{m\ell}=\frac{1}{m}\sum_{k=1}^m u_{k\ell}' \delta_{Y^{'}_k}$  time dependent measures supported on the respective absolutely continuous trajectories $t\mapsto Y_k(t)$ and $t\mapsto Y_k'(t)$. Assume that for every $\rho>0$ there exist constants $L_{{\rho,R}_H}$ and $L_{{\rho,R,M}_G}$ such that
$$
\begin{aligned}
&||H*\mu(t)-H*\nu(t)||_{L^{\infty}(B(0,\rho))}\leq L_{{\rho,R}_H}\mathcal{W}_1(\mu(t),\nu(t))\,,\\
&||G^\ell*\mu_{m\,\ell}(t)-G^\ell*\mu^{'}_{m\,\ell}(t)||_{L^{\infty}(B(0,\rho))} \leq L_{{\rho,R,M}_{G^\ell}}\frac{1}{m}\sum_{k=1}^m|Y_k-Y^{'}_k|\,,
\end{aligned}
$$
where $M$ is the maximum of the absolute value of $u_k$ and $u_k'$.
\end{itemize}
\begin{remark}
Note that the first inequality in Hypothesis (A) is satisfied for all probability measures $\mu$ and $\nu$ absolutely continuous with respect to $\mathcal{W}_1$ and with compact support (see \cite{FornasierPiccoli14}). Moreover the second inequality it is satisfied for example when the controls are positive and their sum is one, namely if $\mu_{m\,\ell}$ and $\mu_{m\,\ell}'$ are probability measures.
\end{remark}
\begin{theorem}
Suppose that Hypothesis (A) holds and let $u\in L^\infty([0,T],\mathcal{U})$ be a given fixed control for \eqref{eq_mu}, and two solutions $(Y^1, \mu^1)$ and $(Y^2,\mu^2)$ of \eqref{eq_mu} relative to the control $u$ and given respective initial data $(Y^{0,i},\mu^{0,i})\in\chi$ with $\mu^{0,i}$ compactly supported, $i=1,2$. Then there exists a constant $C_T>0$ such that
$$
||(Y^1,\mu^1)-(Y^2,\mu^2)||_{\chi}\leq C_T||(Y^{0,1},\mu^{0,1})-(Y^{0,2},\mu^{0,2})||_{\chi}\,.
$$
\end{theorem}
\begin{proof}
First of all, since the control $u$ is the same for both solutions we have that
$$
|Y^1_k(t)-Y^2_k(t)|\leq|Y^{0,1}_k-Y^{0,2}_k|
$$
Thus let us focus on the $\mu$-part of the solution.
Given $\mu^1_m,\,\mu^2_m$, in view of the representation of the solution by means of mass transportation there exist constants $\mathcal{L}>0,\,L_R>0$ and $\rho>0$ such that
$$
\small
\begin{aligned}
\mathcal{W}_1(\mu^1(t),\mu^2(t))&\leq\mathcal{W}_1((\mathcal{T}^{\mu^1,\mu_m^1}_t)_{\sharp}\mu^{0,1},(\mathcal{T}^{\mu^2,\mu_m^2}_t)_{\sharp}\mu^{0,2})\\
&\leq\mathcal{W}_1((\mathcal{T}^{\mu^1,\mu_m^1}_t)_{\sharp}\mu^{0,1},(\mathcal{T}^{\mu^1,\mu_m^1}_t)_{\sharp}\mu^{0,2})+\mathcal{W}_1((\mathcal{T}^{\mu^1,\mu_m^1}_t)_{\sharp}\mu^{0,2},(\mathcal{T}^{\mu^2,\mu_m^2}_t)_{\sharp}\mu^{0,2})\\
&\leq\mathcal{L}\mathcal{W}_1(\mu^{0,1},\mu^{0,2})+\mathcal{W}_1((\mathcal{T}^{\mu^1,\mu_m^1}_t)_{\sharp}\mu^{0,2},(\mathcal{T}^{\mu^2,\mu_m^2}_t)_{\sharp}\mu^{0,2})\\
&\leq\mathcal{L}\mathcal{W}_1(\mu^{0,1},\mu^{0,2})+||(\mathcal{T}^{\mu^1,\mu_m^1}_t)-\mathcal{T}^{\mu^2,\mu_m^2}_t)||_{L^{\infty}(B(0,R))}\\
&\leq \mathcal{L}\mathcal{W}_1(\mu^{0,1},\mu^{0,2})+\int_0^t e^{\mathcal{L}(t-s)}||H*\mu^1(s)-H*\mu^2(s)\\
&\phantom{\leq\mathcal{L}\mathcal{W}_1(\mu^{0,1},\mu^{0,2})+}\qquad+\sum_{\ell=1}^3(G^\ell*\mu^1_{m\,\ell}(s)-G^\ell*\mu^2_{m\,\ell}(s))||_{L^{\infty}(B(0,\rho))}\,ds\\
&\leq \mathcal{L}\mathcal{W}_1(\mu^{0,1},\mu^{0,2})+L_{R,M}\int_0^te^{\mathcal{L}(t-s)}\Bigl(\frac{1}{m}\sum_{k=1}^m|Y^1_k(s)-Y^2_k(s)|\\
&\phantom{\leq\mathcal{L}\mathcal{W}_1(\mu^{0,1},\mu^{0,2})+}\qquad+\mathcal{W}_1(\mu^1(s)-\mu^2(s))\Bigr)\,ds\,,
\end{aligned}
$$
where we first applied the triangular inequality, in the second inequality we used the Lipschitz continuity of the map $\mathcal{T}_t^{\mu^1,\mu_m^1}$ given by Lemma \ref{lemma_trasporto} for $\mu^1=\mu^2$ and $\mu_m^1=\mu_m^2$ and Lemma \ref{lemma66} also for the third inequality. Fourth inequality is ahgain a consequence of Lemma \ref{lemma_trasporto}, and the last one due to hypotesys (A). 
Combining all together we have
$$
\begin{aligned}
||(Y^1,\mu^1)-(Y^2-\mu^2)||_{\chi}\leq& C_0\Big(||(Y^{0,1},\mu^{0,1})-(Y^{0,2},\mu^{0,2})||_{\chi}+\\
&\int_0^t||(Y^1,\mu^1)-(Y^2,\mu^2)||_\chi\,ds\Big)\,,
\end{aligned}
$$
for a suitable constant $C_0$ depending on $\mathcal{L},\,L_{R,M},\,T$.
\end{proof}

\noindent We will now derive the existence of solution of \eqref{eq_mu} by a limit process for $N\to\infty$ where we allow for a variable control $\vecu^N$ depending on $N$

\begin{theorem}
\label{exist_limit}
Assume Hypothesis (A) holds and et $(Y^0,\mu^0)\in\chi$ be given, with $\mu^0$ of bounded support in $B(0,R)$. Define a sequence $(\mu^0_N)_{N\in\mathbb{N}}$ of atomic probability measures equi-compactly supported in $B(0,R)$, such that each $\mu^0_N$ is given by $\mu^0_N=\frac{1}{N}\sum_{i=1}^N\delta_{X_i}$ and $\lim_{N\to\infty}\mathcal{W}_1(\mu^0_N,\mu^0)=0$. Fix now a convergent sequence $(\vecu_N)_{N\in\mathbb{N}}\subset L^\infty([0,T],\mathcal{U})$ of control functions, i.e. $\vecu_N\to\vecu^{*}$ in $L^\infty([0,T],\mathcal{U})$, with $\mathcal{U}$  a nonempty compact subset of $\mathbb{R}^{md}$. For each initial datum $\zeta^0_N=(Y^0,X^0_N)$, denote with $\zeta_N(t)=(Y_N,\mu_N):=(Y_N,X_N)$ the unique solution of the finite dimensional problem \eqref{eq_convolution1} with control $\vecu_N$. Then the sequence $(Y_N,\mu_N)$ converges $C^0([0,T],\chi)$ to some $(Y^*,\mu^*)$ which is a solution of \eqref{eq_mu} with initial data $(Y^0,\mu^0)$ and control $\vecu^*$ in the sense of Definition \ref{def_sol}.
\end{theorem}

\begin{proof}
Since the initial data $\zeta^0_N$ are equi-compactly supported, taking into account Hypothesis (A), and the properties \eqref{2_6} and \eqref{2_7}, the trajectories $\zeta_N=(Y_N,\mu_N)$ are equibounded and equi-Lipschitz continuous in $C^0([0,T],\chi)$. Thus, by an application of the ascoli-Arzel\'a theorem for functions on $[0,T]$ and values in the complete space $\chi$, there exists a subsequence $\zeta_N=(Y_N,\mu_N)$ converging uniformly to a limit $\zeta^*=(Y^*,\mu^*)$, which is also equi-compactly supported in a ball $B(0,R_T)$.\\
We now show that $\zeta^*$ is a solution of \eqref{eq_mu} in the sense of Definition \ref{def_sol}. First of all we observe that $\zeta_N\to\zeta^*$ implies that $Y_N$ converges uniformly to $Y^*$ and $\mu_N$ converges to $\mu^*$ in the Wasserstein metric, i.e.
$$
Y_N\rightarrow Y^*\qquad\text{and}\qquad\mathcal{W}_1(\mu_N(t),\mu^*(t))=0\,.
$$
Thus the first equation is easily satisfied by the uniform convergence of the controls. We are now left with verifying that $\mu^*$ is a solution of the second equation of \eqref{eq_mu}. To this end, let us denote
$$
\mu_{m\ell,N}=\frac{1}{m}\sum_{k=1}^mu_{k\ell,N}\delta_{Y_{k,N}}\qquad\text{and}\qquad \mu_{m\ell,*}=\frac{1}{m}\sum_{k=1}^mu_{k\,\ell}^{*}\delta_{Y_{k}^*}
$$ 
For all $\hat{t}\in[0,T] $and for all $\phi\in C^1_C(\mathbb{R}^d)$ we say that
$$
\langle\phi,\mu_N(\hat{t})-\mu_N(0)\rangle=\int_0^{\hat{t}}\Big[\int_{\mathbb{R}^d}\nabla\phi(x)\cdot w_{H,G,\mu_N,Y_N}(t,x)\,d\mu_N(t,x)\Big]\,dt\,,
$$
which is verified considering the differentiation
$$
\frac{d}{dt}\langle\phi,\mu_N(t)\rangle=\frac{1}{N}\frac{d}{dt}\sum_{i=1}^N\phi(X_i(t))=\frac{1}{N}\sum_{i=1}^N\nabla\phi(X_i(t))\cdot \dot X_i(t)\,,
$$
and applying the substitution for the followers variables $X$ as in \eqref{eq_convolution1}. Moreover
\begin{equation}
\label{weak_form}
\lim_{N\to\infty}\langle\phi,\mu_N(\hat{t})-\mu_N(0)\rangle=\langle\phi,\mu_*(\hat{t})-\mu^0\rangle\,,
\end{equation}
for all $\phi\in C^1_C(\mathbb{R}^d)$, and by Assumption (A) we have that for every $\rho>0$
$$
\lim_{N\to\infty}||(H*\mu_N(t)+\sum_{\ell=1}^dG^\ell*\mu_{m\ell,N}(t))-(H*\mu_*(t)+\sum_{\ell=1}^dG^\ell*\mu_{m\ell,*}(t))||_{L^\infty(B(0,\rho))}=0.
$$
As $\phi\in C^1_C(\mathbb{R}^d)$ has compact support, it follows that
$$
\lim_{N\to\infty}||\nabla\phi\cdot\Big((H*\mu_N(t)+\sum_{\ell=1}^dG^\ell*\mu_{m\ell,N}(t))-(H*\mu_*(t)+\sum_{\ell=1}^dG^\ell*\mu_{m\ell,*}(t))\Big)||_\infty=0.
$$
Denoting with $\Lambda^1_{[0,\hat{t}]}$ the lebesgue measure on $[0,\hat{t}]$, since the product measures $\Lambda^1_{[0,\hat{t}]}\times\frac{1}{\hat{t}}\mu_N(t)$ converge in $\mathcal{P}_1([0,\hat{t}]\times\mathbb{R}^d)$ to $\Lambda^1_{[0,\hat{t}]}\times\frac{1}{\hat{t}}\mu_*(t)$ we finally get
\begin{equation}
\label{limit_mes}
\begin{aligned}
\lim_{N\to\infty}&\int_0^{\hat{t}}\int_{\mathbb{R}^d}\Big(\nabla\phi(x)\cdot \bigl(H*\mu_N(t)+\sum_{\ell=1}^dG^\ell*\mu_{m\ell,N}(t)\bigr)\Big)\,d\mu_N(t)\,dt\\
&=\int_0^{\hat{t}}\int_{\mathbb{R}^d}\Big(\nabla\phi(x)\cdot \bigl(H*\mu_*(t)+\sum_{\ell=1}^dG^\ell*\mu_{m\ell,*}(t)\bigr)\Big)\,d\mu_*(t)\,dt\,.
\end{aligned}
\end{equation}
The thesis follows combining \eqref{weak_form} and \eqref{limit_mes}.
\end{proof}

\subsection{Mean field optimal control}
In this section we focus on a mean field optimal control for the dynamic system with a finite number of controlled agents, in particular when the finite dynamics is a nonlinear affine control system. Here we will give the solution to the infinite optimal control problem as the limit, for the number of followers that goes to infinity, of the solution of the finite dimensional one.
\subsubsection{The finite dimensional optimal control problem}
In the space $\mathbb{X}:=\mathbb{R}^{dm}\times\mathbb{R}^{dN}$ equipped with the norm in \eqref{finite_norm}, given an initial datum\\ $(Y_1(0),\cdots, Y_m(0),X_1(0),\cdots, X_N(0))$, we consider the following optimal control problem
\begin{equation}
\label{finite_opt_control}
\min_{(\vecu_1,\cdots, \vecu_m)\in L^{\infty}([0,T],\mathcal{U})}\int_0^T \Bigg(L(Y(t),X(t))+\sum_{k=1}^m|\vecu_k(t)|^2\Bigg)\,dt\,,
\end{equation}
with $L:\mathbb{X}\to\mathbb{R}$ a continuous function with respect to the distance induced on $\mathbb{X}$ by the norm \eqref{finite_norm}, for example 
\begin{equation}
\label{finite_L}
L(Y(t),X(t)):=\frac{1}{2} \sum_{i=1}^N|X_i(t)-X^*|^2,
\end{equation}
where $X^*$ is a desired configuration of agents, can be used if we want to minimize the distance from a fixed state configuration. With this particular choice of L, it does not depend on $Y(t)$ since we are mainly interested in getting closer to a fixed configuration $X^*$ only in the $X$ variables, since the $Y$ ones are a finite number and directly controlled.\\
Recall that
 $(Y,X)\in\mathbb{R}^{dm}\times\mathbb{R}^{dN}$ are governed by the following affine control equation
\begin{equation}
\label{dyn_constr}
\begin{cases}
\dot Y_k=\vecu_i\quad k=1\ldots m\\
\dot X_i=H*\mu_{N}(X_i)+\sum_{\ell=1}^d G^\ell*\mu_{m\,\ell}(X_i)\quad i=1\ldots N\,,
\end{cases}
\end{equation}
thus we want to find the optimal control that steers the system as close as possible to the desired configuration of the followers $X^*$.\\
First of all we need to prove the existence of a solution.
\begin{theorem}
\label{existence_finite}
The optimal control problem \eqref{finite_opt_control}-\eqref{dyn_constr} with initial datum\\ $(Y_1(0),\cdots, Y_m(0),X_1(0),\cdots, X_N(0))$ has a solution.
\end{theorem}
\begin{proof}
Using the fact that a Caratheodory solution of \eqref{dyn_constr} exists, the fact that the controls are bounded and the continuity of $L$ in $\mathbb{X}$ the statement follows from classical tools in optimal control (Filippov theorem) and the idea of the proof is similar to the one \cite{FornasierPiccoli14}
\end{proof}
The solution of the finite dimensional optimal control problem \eqref{finite_opt_control},\eqref{dyn_constr} can be characterized in the following way.\\
Let us denote by $Z=H^1([0,T],\mathbb{R}^{dm})\times H^1([0,T],\mathbb{R}^{dN}$ the state space of system \eqref{dyn_constr}. Further we define $W=L^2([0,T],\mathbb{R}^{dm})\times L^2([0,T],\mathbb{R}^{dN})$.
\begin{theorem}
Let $(\bar Y_1,\cdots,\bar Y_m,\bar X_1,\cdots,\bar X_N,\bar \vecu)$ be an optimal solution for the optimal control problem \eqref{finite_opt_control},\eqref{dyn_constr}. Then there exist  $(\xi_Y,\xi_X)\in W$ such that the the first order optimality conditions read
\begin{align}
\label{u_finite_opt}
&\bar u_{k\,\ell}=\mathrm{Proj}_{\ \mathcal{U}}\Big(-\bar\xi_{X_i} \cdot (G^\ell*\tilde\mu_{m})-\bar\xi_{Y_k}^\ell\Big)\quad \text{with}\quad \tilde\mu_m=\sum_{k=1}^m\delta_{\bar Y_k}\\
\label{x_finite_opt}
&\begin{cases}
\dot{\bar Y}_k=\bar\vecu_k\quad k=1\ldots m\\
\dot {\bar X}_i=H*\mu_{N}(\bar X_i)+\sum_{\ell=1}^d (G^\ell*\mu_{m\,\ell})(\bar X_i)\quad i=1\ldots N\\
\bar Y(0)=\bar Y_0\\
\bar X(0)=\bar X_0
\end{cases}\\
\label{xi_finite_opt}
&\begin{cases}
\dot{\bar \xi}_{Y_k}=\bar \xi_{X_i}\cdot \partial_{Y_k}\Big[\sum_{\ell=1}^d(G^\ell*\mu_{m\,\ell})(\bar X_i)\Big]
\\
\dot {\bar\xi}_{X_i}=|\bar X_i-X^*|+\bar \xi_{X_i}\cdot \partial_{X_i}\Big[(H*\mu_N)(\bar X_i)+\sum_{\ell=1}^d (G^\ell*\mu_{m\,\ell})(\bar X_i)\Big]\,,
\end{cases}
\end{align}
supplemented by the terminal conditions $\bar{\xi}_{X_i}(T)=0$, $\bar{\xi}_{Y_k}(T)=0$.
\end{theorem}
\begin{proof}
Let $(\bar\xi_{X},\bar\xi_{Y},\eta_X,\eta_Y)\in W\times (\mathbb{R}^{dm}\times \mathbb{R}^{dN}) $ denote the Lagrange multipliers, which are in fact the adjoint states. Then the extended Lagrange functional corresponding to the optimal control problem \eqref{finite_opt_control},\eqref{dyn_constr} is
\begin{equation}
\begin{aligned}
\label{extended_Lagr}
\int_0^T&\tilde L(\bar Y_k,\bar X_i,\vecu,\bar\xi_{X_i},\bar\xi_{Y_k},\eta_X,\eta_Y)\,dt=\int_0^T\Bigg(L(\bar Y,\bar X)+\sum_{k=1}^m|\bar\vecu_k(t)|^2+\\
&\sum_{i=1}^N\Big\langle\bar\xi_{X_i},\dot{\bar X}_i-(H*\mu_{N}(\bar X_i)+\sum_{\ell=1}^d G^\ell*\mu_{m\,\ell}(\bar X_i))\Big\rangle\\
&+\sum_{k=1}^m\Big\langle\bar\xi_{Y_k},
\dot {\bar Y}_k-\bar\vecu_k\Big\rangle\Bigg)\,dt+\eta_X\cdot(\bar X(0)-\bar X_0)+\eta_Y\cdot(\bar Y(0)-\bar Y_0).
\end{aligned}
\end{equation}
As usual the first order optimality conditions are derived solving
$$
d\tilde L(\bar Y_k,\bar X_i,\bar\vecu,\bar\xi_{X_i},\bar\xi_{Y_k},\eta_X,\eta_Y)=0.
$$
The derivative with respect to the adjoint state results in the state equations, while the derivative  with respect to the state gives the adjoint system, while the optimality condition is obtained by the derivative with respect to the control.\\
Indeed formal calculation give that for any $\vech=(h_X,h_Y,h_\vecu)$
\begin{small}
\begin{align}
&d_{\bar u_k^\ell}\tilde L[h_{u_k^\ell}]=\int_0^T\Big(2\bar u_k^{\ell}+\bar\xi_{X_i} \cdot (G^\ell*\tilde\mu_m)(\bar X_i)+\bar\xi_{Y_k}^\ell\Big)h_{u_k^\ell}\,dt,\\
&d_{\bar X_i}\tilde L[h_{X_i}]=\int_0^T\frac{d h_{X_i}}{dt}\cdot \bar\xi_{X_i}+|\bar X_i-X^*|+\\
&\nonumber\phantom{d_{\bar X_i}\tilde L[h_{X_i}]=} \bar\xi_{X_i}\cdot \partial_{\bar X_i}\Big[(H*\mu_N)( \bar X_i)+\sum_{\ell=1}^d (G^\ell*\mu_{m\,\ell})(\bar X_i)\Big]h_{X_i}\,dt+h_{X_i}(0)\cdot \eta_{X_i},\\
&d_{\bar Y_k}\tilde L[h_{Y_k}]=\int_0^T\frac{dh_{Y_k}}{dt}\cdot \bar\xi_{Y_k}+\bar\xi_{X_i}\cdot \partial_{\bar Y_k}\Big[\sum_{\ell=1}^d(G^\ell*\mu_{m\,\ell})(\bar X_i)\Big]h_{Y_k}\,dt+h_{Y_k}(0)\cdot \eta_{Y_k}.
\end{align}
\end{small}
Integrating by parts, taking into account the terminal conditions $\bar{\xi}_{X_i}(T)=0$, $\bar{\xi}_{Y_k}(T)=0$ and that the adjoint variables are in $H^1$ we get the strong formulation of the adjoint system.
\end{proof}
\noindent Thus we have proved what are the necessary conditions that the optimal solution should satisfy.
\subsubsection{The infinite dimensional optimal control problem}
We now consider the infinite dimensional version of the cost functional \eqref{finite_opt_control}
\begin{equation}
\label{infinite_opt_control}
\min_{(\vecu_1,\cdots, \vecu_m)\in L^{\infty}([0,T],\mathcal{U})}\int_0^T \Bigg(\mathcal{L}(Y(t),\mu(t))+\sum_{k=1}^m|\vecu_k(t)|^2\Bigg)\,dt\,,
\end{equation}
with \\
\begin{center}
(L) $\mathcal{L}$ a nonnegative continuous function such that if $(\mu_j)_{j\in\mathbb{N}}$ is  a sequence converging narrowly to $\mu$ then $\mathcal{L}(Y,\mu_j)\to\mathcal{L}(Y,\mu)$ uniformly.
\end{center}
For example
\begin{equation}
\mathcal{L}(Y(t),\mu):=\frac{1}{2}\int_{\mathbb{R}^3}|x-x^*|^2\mu(x,t)\,dx\,,
\end{equation}
if in the finite dimensional case it is \eqref{finite_L}, where $x^*$ represent the desired state.\\
The optimization problem should be solved under the dynamic constraint
\begin{equation}
\label{eq_mu2}
\begin{cases}
Y_k=\vecu_k\qquad \text{for}\,i=1\ldots m\\
\partial_t\mu+\nabla\cdot\Bigl[\bigl(H*\mu+\sum_{\ell=1}^dG^\ell*\mu_{m\,\ell}\bigr)\mu\Bigr]=0\,,
\end{cases}
\end{equation}
where $\mu_{m\,\ell}=\frac{1}{m}\sum_{k=1}^m \delta_{Y_k}u_{k\,\ell}$. This is nothing but the infinite dimensional version of the finite dimensional optimal control problem in which we are trying to minimize the distance from a fixed configuration.\\\\
We can now state the convergence result both for the functionals and for the solutions.
\begin{theorem}
Let $H$ and $G^\ell$, $\ell=1,\cdots d$ satisfying assumption (H) and $\mathcal L$ satisfying assumption (L). Given an initial datum $(Y^0,\mu^0)\in\chi$ and an approximating sequence $\mu^0_N$, with $\mu^0,\,\mu^0_N$ equicompactly supported for all $N\in\mathbb{N}$, then the sequence of functionals
$$
J_N(\vecu)=\int_0^T \Bigg(\mathcal{L}(Y_N(t),\mu_N(t))+\sum_{k=1}^m|\vecu_k(t)|^2\Bigg)\,dt\,,
$$
$\Gamma-$converges to the functional
$$
J(\vecu)=\int_0^T \Bigg(\mathcal{L}(Y(t),\mu(t))+\sum_{k=1}^m|\vecu_k(t)|^2\Bigg)\,dt.
$$
\end{theorem}
\begin{proof}
Let us start by the $\Gamma-\liminf$ condition. Let us fix a sequence of controls $\vecu_N\to\vecu^*$ in $L^\infty$. As done in the proof of Theorem \ref{exist_limit} we can associate with each of these controls a sequence of solutions $\zeta_N(t)=(Y_N(t),\mu_N(t)):=(Y_N(t),X_N(t))$ of \eqref{dyn_constr}, uniformly converging to a solution $\zeta^*(t)=(Y^*(t),\mu^*(t))$ of \eqref{eq_mu2} in the sense of definition \ref{def_sol} with control $\vecu^*$ and initial datum $(Y^0,\mu^0)$. Since solutions $\zeta_N$ and $\zeta^*$ have supports uniformly bounded with respect to $N$ and by the uniform convergence of trajectories $Y_N(t)\to Y^*(t)$ in $L^\infty$ and $\mathcal{W}_1(\mu_n(t),\mu(t))\to 0$, from condition (L) it follows that
\begin{equation}
\label{liminf}
\lim_{N\to+\infty}\int_0^T \mathcal{L}(Y_N(t),\mu_N(t))=\int_0^T \mathcal{L}(Y^*(t),\mu^*(t)).
\end{equation}
Moreover by the $L^\infty$ convergence of $\vecu_N$ to $\vecu^*$ ve have
$$
\liminf_{N\to+\infty}\int_0^T\sum_{k=1}^m|\vecu_{k,N}(t)|^2\,dt\geq\int_0^T\sum_{k=1}^m|\vecu_{k}^*(t)|^2\,dt\,.
$$
Thus we get
$$
\liminf_{N\to+\infty}J_N(\vecu_N)\geq J(\vecu^*)\,.
$$
We now want to address the $\Gamma-\limsup$ condition. Let us fix $\vecu^*$ and consider the trivial sequence $\vecu_N\equiv\vecu^*$ for all $N\in\mathbb{N}$. As we have done before we can associate with each of these controls a sequence of solutions $\zeta_N(t)=(Y_N(t),\mu_N(t)):=(Y_N(t),X_N(t))$ of \eqref{dyn_constr} uniformly converging to $\zeta^*(t)=(Y^*(t),\mu^*(t))$ of \eqref{eq_mu2}  with control $\vecu^*$ and initial datum $(Y^0,\mu^0)$. Therefore we can similarly conclude the limit \eqref{liminf}.
In addition, since the sequence $\vecu_N$ is a constant sequence we have 
$$
\liminf_{N\to+\infty}\int_0^T\sum_{k=1}^m|\vecu_{k,N}(t)|^2\,dt\geq\int_0^T\sum_{k=1}^m|\vecu_{k}^*(t)|^2\,dt\,.
$$
Combyining the two limits we easily get 
$$
\limsup_{N\to+\infty}J_N(\vecu_N)=\lim_{N\to+\infty}J_N(\vecu^*)=J(\vecu^*).
$$
\end{proof}
\begin{corollary}
Let $H$ and $G^\ell$ $\ell=1,\cdots d$ satisfying assumption (H) and $\mathcal{L}$ satisfy hypothesis (L). Given an initial datum $(Y^0,\mu^0)\in\chi$ with $\mu^0$ compactly supported, the optimal control problem
$$
\min_{(\vecu_1,\cdots, \vecu_m)\in L^{\infty}([0,T],\mathcal{U})}\int_0^T \Bigg(\mathcal{L}(Y(t),\mu(t))+\sum_{k=1}^m|\vecu_k(t)|^2\Bigg)\,dt\,,
$$
has solution, where $(Y,\mu)$ defines the solution of \eqref{eq_mu2} with initial datum $(Y^0,\mu^0)$ and control $\vecu$. Moreover this solution can be built as limits $\vecu^*$ of sequences of optimal controls $\vecu_N^*$ of the finite dimensional problems 
$$
\min_{(\vecu_1,\cdots, \vecu_m)\in L^{\infty}([0,T],\mathcal{U})}\int_0^T \Bigg(\mathcal{L}(Y_N(t),\mu_N(t))+\sum_{k=1}^m|\vecu_k(t)|^2\Bigg)\,dt\,,
$$
where $\mu_N=\frac{1}{N}\sum_{i=1}^N\delta_{X^i_N}$ and $\mu_{m,N}^\ell=\frac{1}{m}\sum_{k=1}^mu_k^\ell\delta_{Y_k}$ are measures supported  on the trajectories of the system
\begin{equation}
\label{dyn_constr2}
\begin{cases}
\dot{ Y}_k=\vecu_k\quad k=1\ldots m\\
\dot { X}_i=H*\mu_{N}( X_i)+\sum_{\ell=1}^d (G^\ell*\mu_{m\,\ell})( X_i)\quad i=1\ldots N\,,
\end{cases}
\end{equation}
with initial datum $(Y^0,X^0_N)$, control $\vecu$ and $\mu^0_N$ is such that $\mathcal{W}_1(\mu^0_N,\mu^0)\to0$ for $N\to+\infty$.
\end{corollary}
\begin{proof}
First of all note that the optimal controls $\vecu^*_N$ of the finite dimensional optimal control problem belong to $L^\infty([0,T],\mathcal{U})$ hence $\vecu_N^*$ admits a subsequence converging to some $\vecu^*\in L^\infty([0,T],\mathcal{U})$. As we have done previously in the proof of Theorem \ref{exist_limit}, we can associate with each of these controls a sequence of solutions $\zeta_N(t)=(Y_N(t),\mu_N(t))$ of \eqref{dyn_constr2} uniformly converging to a solution $\zeta^*(t)=(Y(t),\mu(t))$ in the sense of Definition \ref{def_sol} with control $\vecu^*$. In order to conclude that $\vecu^*$ is an optimal control we need to show that it is actually a minimizer of $J$. Let $\vecu$ be an arbitrary control and $\vecu_N$ be a recovery sequence given by the $\Gamma-\limsup$ condition, so that 
$$
J(\vecu)\geq\limsup_{N\to+\infty}J_N(\vecu_N)\,.
$$
By using the optimality of $\vecu_N^*$ we have
$$
\limsup_{N\to+\infty} J_N(\vecu_N)\geq\limsup_{N\to+\infty} J_N(\vecu_N^*)\geq\liminf_{N\to+\infty} J_N(\vecu_N^*)\,,
$$
and applying the $\Gamma-\liminf$ condition we get
$$
\liminf_{N\to+\infty} J_N(\vecu_N^*)\geq J(\vecu^*)\,\,\Rightarrow\,\,J(\vecu)\geq J(\vecu^*).
$$
\end{proof}
\noindent After proving the convergence of the optimal solution for the finite dimensional optimal control problem to the infinite dimensional one, we are interested in finding necessary conditions that the optimal solution should satisfy. Also in this case we start from the necessary conditions for the finite dimensional one recovering then the infinite dimensional ones.\\
Let us denote by $\mathcal{Z}=H^1([0,T],\mathbb{R}^{dm})\times H^1([0,T],L^2(\mathbb{R}^{dN}))$ the state space of system \eqref{eq_mu2}. Further we define $$\mathscr{W}=L^2([0,T],\mathbb{R}^{dm})\times\bigl(H^1([0,T],L^2(\mathbb{R}^{dN}))\cap L^2([0,T],H^1(\mathbb{R}^{dN}))\bigr)\,.$$
\begin{theorem}
Let $(\bar Y, \bar\mu,\bar\vecu)$ be an optimal solution for the infinite dimensional optimal control problem \eqref{infinite_opt_control}\eqref{eq_mu2}. Then there exist $(\bar\phi_Y,\bar\phi_\mu)\in \mathscr{W}$ such that the optimality conditions read
\begin{align}
\label{u_infinite_opt}
&\bar \vecu_k^{\ell}=\mathrm{Proj}_{\ \mathcal{U}}\Big(-\bar\phi_{\mu} \cdot (G^\ell*\tilde\mu_m)-\bar\phi_{Y_k}^\ell\Big)\quad \text{with}\quad \tilde\mu_m=\sum_{k=1}^m\delta_{\bar Y_k}\,,\\
\label{mu_infinite_opt}
&\begin{cases}
\dot{\bar Y}_k=\bar \vecu_k\qquad \text{for}\,\,i=1\ldots m\\\
\partial_t\bar \mu+\nabla\cdot\Bigl[\bigl(H*\bar\mu+\sum_{\ell=1}^dG^\ell*\tilde\mu_{m\,\ell}\bigr)\bar\mu\Bigr]=0\\
\bar Y(0)=\bar Y_0\\
\bar\mu(0)=\bar \mu_0
\end{cases}\\
\label{adj_infinite_opt}
&\begin{cases}
\partial_t\bar\phi_\mu=|x-x^*|-\nabla\bar\phi_\mu\cdot\Big[(H*\bar\mu+\sum_{\ell=1}^dG^\ell*\tilde\mu_{m\,\ell})+D_\mu((H*\bar\mu)\bar\mu)\Big]\\
\dot{\bar \phi}_{Y_k}=\bar \phi_\mu\cdot \partial_{\bar Y_k}\Big[\sum_{\ell=1}^d(G^\ell*\tilde\mu_{m\,\ell})\Big]\,.
\end{cases}
\end{align}
Supplemented by the terminal conditions $\bar\phi_\mu(T,\cdot)=0$ and $\bar\phi_{Y_k}(T)=0$.
\end{theorem}
\begin{proof}
The proof is similar to the one for the finite dimensional case, using the extended Lagrange functional
$$
\small
\begin{aligned}
&\int_0^T\tilde{\mathcal{L}}(\bar Y,\bar\mu,\bar\vecu,\bar\phi_{\mu},\bar\phi_{Y},\eta_\mu,\eta_Y)\,dt=\int_0^T \Bigg(\mathcal{L}(\bar Y(t),\bar\mu(t))+\sum_{k=1}^m|\bar \vecu_k(t)|^2+(\dot \bar Y_k-\bar\vecu_k)\cdot\bar\phi_{Y_k}\\
&+\int_{\mathbb{R}^d}\bigl(\partial_t\bar\phi_\mu+\nabla\bar\phi_\mu\cdot(H*\bar\mu+\sum_{\ell=1}^dG^\ell*\tilde\mu_{m\,\ell})\bigr)\bar\mu\,dx\\
&+\int_{\mathbb{R}^d}\bar\phi_\mu(T,x)\bar\mu(T,x)-\bar\phi_\mu(0,x)\bar\mu(0,x)\,dx-\int_{\mathbb{R}^d}(\bar\mu(0)-\bar\mu_0)\eta_\mu\,dx+(\bar Y(0)-\bar Y_0)\cdot \eta_{Y}\,,
\end{aligned}
$$
where $(\bar\phi_\mu,\bar\phi_Y,\eta_\mu,\eta_Y)\in \mathscr{W}\times \Bigl(L^2(\mathbb{R}^{dN})\times\mathbb{R}^{dm}\Bigr)$ are Lagrange multipliers or coadjoint variables. Analogously to the finite dimensional case we derive the adjoint system and the optimality conditions by calculating the derivatives of $\tilde{\mathcal{L}}$ with respect to the state variable $(\bar Y,\bar\mu)$ in direction $ (h_Y,h_\mu)$ and the control $\bar\vecu$ in direction $h_\vecu$.
\end{proof}

\section{Kinetic limit}
\label{sec:kinetic}
In this section, differently from what we have done in the previous section, we will investigate what happens when the number of agents goes to infinity preserving the percentage of controlled particles. Therefore the number of controlled agents will not remain finite but goes to infinity as the non controlled ones. To this end we will use a Boltzmann approach.
Let us consider $\alpha$ as a characteristic time of interaction, supposing that only two agents $x$ and $y$ are interacting with each other, we have that the state $\hat x$ of the particle $x$ after the interaction with the particle $y$ is an affine function of the controls
\begin{equation}
\label{Kinetic_inter}
\small
\begin{aligned}
\hat x=x+\alpha H(x-y)(1-\Theta)(1-\Theta^*)+\alpha\sum_{\ell=1}^d \Big[G^\ell(x-y)(u_{\alpha}^{*})_\ell(1-\Theta)\Theta^*+\vecu_\alpha\Theta\Big]\,,
\end{aligned}
\end{equation}
where $\vecu_\alpha$ and $\vecu^*_\alpha$ are measuarable functions of time, $\Theta$ is a random variable which characterizes the probability that the agent $x$ is controlled through the control $\vecu_\alpha$ and $\Theta^*$ a random variable which characterizes the probability that the agent $y$ is controlled via the control $\vecu^*_\alpha$. More precisely their law is a Bernoulli one, so that we have
$$
\mathcal{P}\{\Theta=1\}=p\in[0,1]\qquad\mathcal{P}\{\Theta^*=1\}=p\in[0,1].
$$
Thus we now show which is the continuity equation that the density of agents should satisfy, starting from a dynamics affine in the controls.\\
Consider a kinetic model for the evolution of the density $\mu=\mu(x,t)$ of agents, with $x\in\mathbb{R}^d$ at time $t$, ruled by the following Boltzmann-type equation
\begin{equation}
\label{Boltzman1}
\partial_t \mu(x,t)=\mathbb{Q}_\alpha(\mu,\mu)(x,t)\,,
\end{equation}
where $\mathbb{Q}_\alpha$ is an interaction operator which accounts the gain and loss of agents in position $x$.
\begin{equation}
\label{interaction_kernel}
\mathbb{Q}_\alpha=\eta\Big\langle\int_{\mathbb{R}^d} \frac{1}{\mathcal{J}_\alpha}\mu(\hat x,t)\mu(\hat y,t)-\mu(x,t)\mu(y,t)\,dy\Big\rangle\,,
\end{equation}
where $\mathcal{J}_\alpha$ represents the Jacobian of the transformation $(x,y)\to(\hat x,\hat y)$.\\
In order to state the main result, namely performing the quasi invariant limit for interaction strength low and frequency high, for a binary dynamics affine in the controls, we give the following definition
\begin{definition}
The density $\mu$ is called weak solution of the Boltzmann equation \eqref{Boltzman1}, with initial datum $\mu_0\in\mathcal{P}_1(\mathbb{R}^d)$ if $\mu\in L^2([0,T],\mathcal{P}_1(\mathbb{R}^d))$ such that $\mu(x,0)=\mu_0$ and $\mu$ satisfies the weak form of \eqref{Boltzman1}, i.e
$$
\frac{d}{dt}\bigl\langle\mu,\varphi\bigr\rangle=\bigl\langle\mathbb{Q}_\alpha(\mu,\mu),\varphi\bigr\rangle\,,
$$
for all $t\in[0,T]$ and $\varphi\in C^\infty_C(\mathbb{R}^d,\mathbb{R})$ and 
where
$$
\bigl\langle\mathbb{Q}_\alpha(\mu,\mu),\varphi\bigr\rangle=\Big\langle\int_{\mathbb{R}^d}\int_{\mathbb{R}^d}(\varphi(\hat x)-\varphi(x))\mu(x)\mu(y)\,dx\,dy\Big\rangle.
$$
\end{definition}
We are now ready to state the main result, namely to prove which is the infinite dimensional equation that the density of agents satisfy when the interaction strength is low and the frequency is high
\begin{theorem}
\label{K_limit}
Let us fix controls $\vecu_\alpha$ and $\vecu^*_\alpha$, measurable functions of time which stay in a compact set of $\mathbb{R}^{md}$ and suppose that
\begin{align*}
&\lim_{\alpha\to 0} \vecu_\alpha(t)=\bar\vecu(t)\,,\\
&\lim_{\alpha\to 0} \vecu^*_\alpha(t)=\bar\vecu^*(t).
\end{align*} 
Consider a weak solution $\mu$ of \eqref{Boltzman1} with initial datum $\mu_0(x)$. Thus introducing the following scaling
$$
\alpha=\epsilon\qquad \eta=\frac{1}{\epsilon}
$$
for the binary interaction \eqref{Kinetic_inter}, and defining by $\mu^\epsilon(x,t)$ a solution for the scaled equation \eqref{Boltzman1}, for $\epsilon\to 0$ $\mu^\epsilon(x,t)$ converges pointwise, up to a subsequence, to $\mu(x,t)$ where $\mu$ satisfies
\begin{equation}
\label{PDE}
\partial_t \mu+\Div\Big[\Big(\int_{\mathbb{R}^3}\Big\langle K(x,y,\bar\vecu,\bar\vecu^*)\Big\rangle \mu(y,t)\,dy\Big)\mu\Big]=0\,,
\end{equation}
where 
\begin{equation}
\label{K}
\Big\langle K(x,y,\bar\vecu,\bar\vecu^*)\Big\rangle =(1-p)^2H(x-y)+p(1-p)\sum_{\ell=1}^dG^\ell(x-y)\bar u_{\ell}^{*}+\bar\vecu p\,.
\end{equation}
\end{theorem}
\begin{proof}
The proof follows the ideas of the one in \cite{AlbiFornasier16} but now the binary dynamics is a bit different, i.e. it is an affine function of the controls.\\
Since $\mu=\mu(x,t)$ is a weak solution of \eqref{Boltzman1}, pick a test function $\varphi\in C^\infty_c(\mathbb{R}^d,\mathbb{R})$ we want to compute
\begin{equation}
\label{weak_eq}
\frac{d}{dt}\int_{\mathbb{R}^d}\varphi(x(t))\mu(x(t),t)\,dx=\frac{1}{\alpha}\Big\langle\int_{\mathbb{R}^d}\int_{\mathbb{R}^d}(\varphi(\hat x)-\varphi(x))\mu(x,t)\mu(y,t)\,dx\,dy\Big\rangle\,.
\end{equation}
For $\alpha=\epsilon$ small, using Taylor expansion we have
\begin{equation}
\label{phi_diff}
\varphi(\hat x)-\varphi(x)=\nabla\varphi(x) (\hat x-x)+\frac{1}{2}\mathscr{H}(x)(\hat x-x)\cdot(\hat x-x)+o(|\hat x-x|^2)\,.
\end{equation}
Taking $\bar\vecu$ and $\bar\vecu^*$ in a compact set, using the boundedness of the vector field in \eqref{Kinetic_inter} and substituting \eqref{phi_diff} in \eqref{weak_eq}, we can neglect the term $\frac{1}{2}\mathscr{H}(x)(\hat x-x)\cdot(\hat x-x)$ as $\epsilon\to 0$ since it is bounded and, in \eqref{weak_eq}, is of order $\epsilon$, which goes to zero as $\epsilon\to 0$. Therefore we obtain
\begin{equation}
\small
\frac{d}{dt}\int_{\mathbb{R}^d}\varphi(x(t))\mu(x(t),t)\,dx=\int_{\mathbb{R}^d}\nabla\varphi(x)\cdot\Big[\int_{\mathbb{R}^d}\Big\langle K(x,y,\bar\vecu,\bar\vecu^*)\Big\rangle \mu(y,t)\,dy\Big]\mu(x,t)\,dx\,,
\end{equation}
where is the right hand side of \eqref{Kinetic_inter}, i.e.
$$
\begin{aligned}
K(x,y,\vecu,\vecu^*)=H(x-y)(1-\Theta)(1-\Theta^*)+\sum_{\ell=1}^dG^\ell(x-y)\bar u_{\ell}^{*}(1-\Theta)\Theta^*+\bar\vecu\Theta\,.
\end{aligned}
$$
This equation turns out to be the weak formulation of
\begin{equation}
\partial_t \mu+\Div\Big[\Big(\int_{\mathbb{R}^d}\Big\langle K(x,y,\bar\vecu,\bar\vecu^*)\Big\rangle \mu(y,t)\,dy\Big)\mu\Big]=0.
\end{equation}
\end{proof}
\begin{remark}
Note that splitting \eqref{PDE}, it reads also
\begin{equation}
\label{PDE2}
\partial_t \mu+(1-p)^2\Div\bigl[(H*\mu)\mu\bigr]+
p(1-p)\sum_{\ell=1}^du_{\ell}^{*}\Div\bigl[(G^{\ell}*\mu)\mu\bigr]+p\Div(\vecu \mu) =0.
\end{equation}
\end{remark}

\subsection{The optimal control problem: binary control and Boltzmann approach}
In this section, starting from the finite dimensional solution of an optimal control problem, derive the infinite dimensional one, exploiting a Boltzmann approach. We will give a sub optimal solution to the infinite dimensional optimal control problem starting from the optimal solution of the finite dimensional one
\subsubsection{The binary optimal control}
Let us consider the evolution equation \eqref{ODE_N}, with the control variables $\vecu$ with values in a compact $U\subset\mathbb{R}^{dm}$. The controllers are obtained as the
solution of the following optimal control problem
\begin{equation}
\label{opt_finite_1}
\min_{\vecu\in L^\infty([0,T],\mathcal{U})}\mathcal{J}(\vecu,X_0)=\int_0^Te^{-\lambda t}\mathscr{L}(X(t),\vecu(t))\,dt\,,
\end{equation}
with
$$
\mathscr{L}(X(t),\vecu(t))=||\bar{X}-X(t)||_2^2+\gamma||\vecu||_2^2\,,
$$
with $\bar X$ a desired fixed configuration.\\
We first focus our analysis on the optimal control problem when $N=2$, so there are only two agents $i$ and $j$. We shall assume that system dynamics have been discretized in time in intervals $\Delta t=[t_k,t_{k+1}]$ with a first-order approximation

\begin{equation}
\begin{cases}
X_i^{k+1}=X_i^k+\frac{\Delta t}{2}\Big(H(X^k_i-X^k_j)(1-\Theta_i)(1-\Theta_j)+\\
\phantom{X_i^{k+1}=}\Theta_j(1-\Theta_i)\displaystyle{\sum_{\ell=1}^d} G^\ell(X^k_i-X^k_j)u_{j\,\ell}^{k}+\vecu_i^k\Theta_i\Big)\\
X_j^{k+1}=X_j^k+\frac{\Delta t}{2}\Big(H(X^k_j-X^k_i)(1-\Theta_j)(1-\Theta_i)+\\
\phantom{X_j^{k+1}=}\Theta_i(1-\Theta_j)\displaystyle{\sum_{\ell=1}^d }G^\ell(X^k_j-X^k_i)u_{i\,\ell}^{k}+\vecu_j^k\Theta_j\Big)\,,
\end{cases}
\end{equation}
where $H$ and $G^\ell$ $\ell=1,\cdots d$ satisfy assumption $(H)$, $\Theta_i$ is a Bernoulli random variable which describes the probability that the agent $i$ is controlled and we have assumed that the control is piecewise constant in each time interval $[t_k,t_{k+1}]$, i.e $\vecu_i(t)=\sum_{k=1}^{N_T}\vecu_i^k\chi_{[t_k,t_{k+1}]}$.\\
In this setting the discretized functional becomes
\begin{equation}
\min_{\vecu_{ij}1\in L^\infty([0,T],\mathcal{U})}\mathcal{J}_{\Delta t}(\vecu_{ij},X^0_{ij})=\min_{\vecu_{ij}1\in L^\infty([0,T],\mathcal{U})}\sum_{k=0}^{N_T}\beta^k\mathscr{L}(X^k_{ij},\vecu_{ij}^k)\,.
\end{equation}
The strategy is now to find the values of $\vecu_{ij}^k=(\vecu_i^k,\vecu_j^k)$ in each time interval solving an instantaneous optimal control problem. For example in the time interval $[0,t_1]$
the instantaneous control problem is further simplified to minimizing the functional
$$
\mathcal{J}_{\Delta t}(\vecu_{ij},X^0_{ij})=\frac{\beta}{2}||\bar{X}-X_{ij}^1||_2^2+\gamma||\vecu^0||_2^2\,,
$$
with $X_{ij}^1$ given by
\begin{equation}
\label{discrete_evol}
\begin{cases}
X_i^1=X_i^0+\frac{\Delta t}{2}\Big(H(X^0_i-X^0_j)(1-\Theta_i)(1-\Theta_j)+\\
\phantom{X_i^1=}\Theta_j(1-\Theta_i)\displaystyle{\sum_{\ell=1}^d}G^\ell(X^0_i-X^0_j)u_{j\,\ell}^{0}+\vecu_i^0\Theta_i\Big)\\
X_j^1=X_j^0+\frac{\Delta t}{2}\Big(H(X^0_j-X^0_i)(1-\Theta_j)(1-\Theta_i)+\\
\phantom{X_j^1=}\Theta_i(1-\Theta_j)\displaystyle{\sum_{\ell=1}^d}G^\ell(X^0_j-X^0_i)u_{i\,\ell}^{0}+\vecu_j^0\Theta_j\Big)
\end{cases}
\end{equation}
Since $X_{ij}^1$ depend linearly on $(\vecu_i^0,\vecu_j^0)$ the functional is clearly convex and the optimal control is the solution of the following linear system
$$
\vecD_{\Delta t} \begin{pmatrix}\vecu_i^0\\\vecu_j^0\end{pmatrix}=\vecC_{\Delta t}\,,
$$
where $\vecD_{\Delta t}$ is a $2d\times 2d$ square matrix and $\vecC_{\Delta t}$ a vector in $\mathbb{R}^{2d}$.\\
The solution $\vecU^{*^0}_{\Delta t}=(\vecu_i^0,\vecu_j^0)$ is 
\begin{align}
\label{opt_inst_control}
\begin{pmatrix}\vecu_i^0\\\vecu_j^0\end{pmatrix}=\Pi_\mathcal{U}\Big((\vecD_{\Delta t})^{-1}\vecC_{\Delta t}\Big)\,,
\end{align}
where $\Pi_\mathcal{U}$ is the projection of the compact $\mathcal{U}$. Iterating his procedure in each time interval we generate an optimal control in feedback form, i.e. at a given discrete instant $k$, the instantaneous optimal action is a non linear mapping only depending on the current state $X^k$ and model parameters.
\subsubsection{The Boltzmann approach for the infinite dimensional optimal control}
For a large ensemble of agents, the microscopic optimal control problem is well-approximated by the following mean-field optimal control problem,
\begin{equation}
\min_{\vecu\in L^\infty([0,T],\mathcal{U})}J(\mu,\vecu):=\min_{\vecu\in L^\infty([0,T],\mathcal{U})}\int_0^T\int_{\mathbb{R}^d}e^{-\lambda t}\mathcal{L}(X,t, \mu(X,t),\vecu(X,t))\,d\mu(X,t)\,dt\,,
\end{equation}
constrained to the mean-field multi-agent dynamical system for the agents' density distribution $\mu$
\begin{equation}
\label{PDE_mu}
\partial_t \mu+\Div\Big[\Big(\int_{\mathbb{R}^d}\Big\langle K(X,Y,\vecu(X),\vecu(Y))\Big\rangle \mu(Y,t)\,dY\Big)\mu\Big]=0.
\end{equation}
with 
$$
\Big\langle K(X,Y,\vecu,\vecu^*)\Big\rangle =(1-p)^2H(X,Y)+p(1-p)\sum_{\ell=1}^d G^\ell(X,Y)u^{*}_{\ell}+\vecu p\,,
$$
where $p$ is the probability that the particle $X$ or $Y$ is controlled. 
 The mean-field functional $J(\mu,\vecu)$ is defined accordingly to the finite dimensional cost $\mathscr{L}$ in this way:
\begin{equation}
\label{opt_infinite_1}
J(\mu,\vecu):=\int_0^T\Bigg(\int_{\mathbb{R}^d}||\bar{X}-X||_2^2\mu(X,t)\,dX+\int_{\mathbb{R}^d}\gamma||\vecu||_2^2\mu(X,t)\,dX\Bigg)\,dt\,,
\end{equation}
where $\bar X$ is the desired fixed configuration.\\
We propose a sub-optimal solution to this problem using a Boltzmann-type equation to model the evolution of a system of agents ruled by a binary interactions.\\ 
For $\mu=\mu(X,t)$ denoting the kinetic probability density of agents in position $X\in\mathbb{R}^d$ at time $t>0$, the time evolution of the density  $\mu$ is given as a balance between bilinear gains and losses of the agents' position, due to the following constrained binary interaction,
\begin{equation}
\label{opt_feedback_dyn}
\begin{cases}
X^*=X+\alpha\Big(H(X-Y)(1-\Theta_X)(1-\Theta_Y)+\\
\phantom{X^*=}\Theta_Y(1-\Theta_X)\displaystyle{\sum_{\ell=1}^d}G^\ell(X-Y)U^{Y*}_{\alpha_\ell}(Y,X)+\vecU^{X*}_\alpha(X,Y)\Theta_X\Big)\\
Y^*=Y+\alpha\Big(H(Y-X)(1-\Theta_Y)(1-\Theta_X)+\\
\phantom{Y^*=}\Theta_X(1-\Theta_Y)\displaystyle{\sum_{\ell=1}^d}G^\ell(Y-X)U^{X*}_{\alpha_\ell}(X,Y)+\vecU^{Y*}_\alpha(Y,X)\Theta_Y\Big)\,,
\end{cases}
\end{equation}
where $(X^*,Y^*)$ are the post-interaction states, the parameter $\alpha$ measures the strength of the interaction and the feedback $\vecU^{i*}_\alpha(X,Y)$ for $i=X,Y$, indicates the control of the dynamics. For $\alpha=\frac{\Delta t}{2}$ and $\vecU^{i*}_\alpha(X^1_i,X^1_j)$ given by \eqref{opt_inst_control} the resulting dynamics is equivalent to the expression \eqref{discrete_evol}.\\
We now proceed similarly to Theorem \ref{K_limit}, making the quasi-invariant limit supposing that the density $\mu$ satisfies a Boltzmann-type equation and the agents dynamic is the one in \eqref{opt_feedback_dyn} considering a regime where interaction strength is low and frequency high.
More precisely we consider that  the density $\mu$ satisfies the Boltzmann type equation \eqref{Boltzman1}
with the interaction operator $\mathbb{Q}_\alpha$ given by \eqref{interaction_kernel}
where $\mathbb{J}_\alpha$ represents the Jacobian of the transformation $(X,Y)\to(X^*,Y^*)$ given in \eqref{opt_feedback_dyn}.
\begin{theorem}
For $\alpha\geq0$ and $t\geq 0$, assume that $H(\cdot)$ and $G(\cdot)\in L^2_{loc}$ and for $\alpha\to 0$ we assume that $\vecU^{i*}_\alpha(X,Y)\to K^i(X,Y)$ for $i=1,2$. Then we consider a weak solution $\mu$ of the Boltzmann-type equation \eqref{Boltzman1} with initial datum $\mu_0(X)$. Setting $\alpha=\epsilon$ and $\eta=\frac{1}{\epsilon}$ for the binary interaction \eqref{opt_feedback_dyn}, and defining $\mu^\epsilon(X,t)$ a solution for the scaled equation \eqref{Boltzman1}, when $\epsilon\to 0$, $\mu^\epsilon(X,t)$ converges point-wise, up to a subsequence, to $\mu(X,t)$ satisfying the following nonlinear equation
\begin{equation}
\label{PDE3}
\partial_t \mu+\Div\Big[\Big(\int_{\mathbb{R}^d}\Big\langle K(X,Y,K^1(X,Y),K^2(X,Y))\Big\rangle \mu(Y,t)\,dY\Big)\mu\Big]=0\,,
\end{equation}
with 
$$
\begin{aligned}
K(X,Y,K^1(X,Y),&K^2(X,Y))=H(X-Y)(1-\Theta_X)(1-\Theta_Y)\\
&+\sum_{\ell=1}^dG^\ell(X-Y)K^{2}_{\ell}(X,Y)(1-\Theta_X)\Theta_Y+K^1(X,Y)\Theta_X\,.
\end{aligned}
$$
\end{theorem}
\begin{proof}
The proof is similar to the one of Theorem \ref{K_limit} in Section \ref{sec:kinetic} and follows the one in \cite{AlbiFornasier16}
\end{proof}
\noindent This is the suboptimal solution to the infinite dimensional optimal control problem.

\section{Comparison between the two approaches}
In this section we make some comparisons between the infinite dimensional dynamics obtained by the kinetic limit and the mean filed granular one. First of all note that in both cases the finite dimensional dynamics in the case $N=2$ is the same, therefore passing to the limit for the number of agents tending to infinity, the two infinite dimensional equations are similar. More precisely looking at  the PDE \eqref{PDE_mu}
$$
\partial_t \mu+\Div\Big[\Big(\int_{\mathbb{R}^d}\Big\langle K(X,Y,\vecu(X),\vecu(Y))\Big\rangle \mu(Y,t)\,dY\Big)\mu\Big]=0.
$$
with 
$$
\Big\langle K(X,Y,\vecu,\vecu^*)\Big\rangle =(1-p)^2H(X,Y)+p(1-p)\sum_{\ell=1}^d G^\ell(X,Y)u^{*}_{\ell}+\vecu p\,,
$$
and at the second PDE equation in \eqref{eq_mu}, namely 
$$
\partial_t\mu+\nabla\cdot\Bigl[\bigl(H*\mu+\sum_{\ell=1}^dG^\ell*\mu_{m\,\ell}\bigr)\mu\Bigr]=0\,.
$$
\noindent the fields of which we are making the divergence are very similar, they both contain the convolution of the density $\mu$ with the kernel $H$, which describes the non interacting dynamics. Moreover in them both there is the action of the kernel $G^\ell$ which describes the interaction between the controlled and non controlled agents.\\
However, we notice that despite this similarity it is not possible to recover one PDE starting from the other. Indeed if in the kinetic framework we consider to control only a finite number of agents independently from the total number, it is not possible to choose a suitable bernoulli parameter $p$ to describe this situation. Controlling a finite number of agents would mean to make $p$ tending to zero. Unfortunately in this limit the resulting PDE does not coincide with \eqref{eq_mu}, since the interaction term $G^\ell$ would disappear making the system not controlled.\\
Nevertheless the two approaches are both valid depending of the problem we are interested in. For example if in the finite dimensional system we are interested in controllability issues, i.e the possibility of finding controls that move the system between two fixed configurations, is reasonable to need a number of controlled agents which is proportional to their total number. If we want to preserve this property also in the limit, the kinetic approach surely would be the more suitable. In this situation, if the finite dimensional system is controllable when the number of controls is a fraction of the total number of agents we are sure, in the finite dimensional optimal control problem \eqref{opt_finite_1} 
with the current cost $\mathscr{L}(X(t),\vecu(t))=||\bar{X}-X(t)||_2^2+\gamma||\vecu||_2^2$, to find a control which allow us to reach the desired configuration $\bar X$. Using the kinetic limit we preserve the fraction of controlled agents keeping the controllability property, crucial to ensure that we are able to reach also in the infinite dimensional case the desired configuration $\bar X$. This kind of problems are typical for swarms of robots or animal collective dynamics.\\ On the contrary in situations in which in the finite dimensional system the same amount of controls are sufficient to control all the others or in situations in which is preferable to identify only a finite number of leaders to convince all the others to have a certain behavior, like in opinion dynamics or crowd motion, the mean field granular approach is more suitable. In this case in the infinite dimensional limit the number of controlled agents remains finite, of course we cannot ensure to be able to reach exactly the desired configuration but only to move towards it, however we use a little amount of control energy.

\section{Conclusions}
In this paper we studied an individual-based model for a multi agent system. The originality of the work is the fact that the individual dynamics is a nonlinear affine control system, where the controls are the velocities of the coordinates of some of the agents and the interaction force between controlled and non-controlled agents occurs through the velocity of the leader agents, like in systems of particles immersed in Stokes fluid. We analyze what happen when the number of agents goes to infinity and  we studied two different situations. The first one in which the number of controlled agents is kept fixed and finite, whereas the non-controlled or passive ones goes to infinity. The second instead where the overall number of agents goes to infinity preserving the percentage of controlled ones.  In both cases we perform a limit procedure, using a granular mean-field limit for the first situation and  a kinetic approach for the second one, and obtain a PDE for the evolution of the agents' density. 
Moreover in both cases we state the optimal control problem starting from the solution of the finite dimensional one and derive the corresponding first order optimality conditions.



\end{document}